\documentclass[a4paper]{article}

\usepackage[latin1]{inputenc}
\usepackage[french,british]{babel}
\usepackage[T1]{fontenc}
\usepackage{setspace}
\usepackage{indentfirst} 
\usepackage{fancyhdr} 
\usepackage{nicefrac}
\usepackage{textcomp}
\usepackage{lastpage}
\usepackage{amssymb,amsfonts,amsmath,amsthm}
\usepackage{latexsym}
\usepackage{url}
\usepackage{verbatim}
\usepackage[all]{xy}
\usepackage{mathrsfs}
\usepackage{stmaryrd}
\usepackage{oldgerm}
\usepackage{bm}
\usepackage{graphics}
\usepackage{wrapfig}
\usepackage{fancybox}
\usepackage{graphicx}
\input xy
\xyoption{all}

\theoremstyle{plain}
\newtheorem{theo}{Theorem}[subsection]
\newtheorem*{theo*}{Theorem}
\newtheorem{lem}[theo]{Lemma}
\newtheorem*{lem*}{Lemma}
\newtheorem{cor}[theo]{Corollary}
\newtheorem{prop}[theo]{Proposition}
\newtheorem*{prop*}{Proposition}

\newtheorem*{conj*}{Conjecture}

\theoremstyle{remark}
\newtheorem{rem}[theo]{Remark}

\theoremstyle{definition}
\newtheorem{defi}[theo]{Definition}

\numberwithin{equation}{subsection}

\DeclareMathAlphabet{\mathpzc}{OT1}{pzc}{m}{it}
\DeclareMathOperator{\NN}{\mathbb{N}}

\DeclareMathOperator{\QQ}{\mathbb{Q}}

\DeclareMathOperator{\HH}{\mathbb{H}}
\DeclareMathOperator{\aA}{\mathbb{A}}

\DeclareMathOperator{\OO}{\mathscr{O}}

\DeclareMathOperator{\Spec}{Spec}

\DeclareMathOperator{\EE}{\mathbf{E}}

\DeclareMathOperator{\edg}{\mathscr{E}}

\DeclareMathOperator{\F}{\mathscr{F}}

\DeclareMathOperator{\f}{F}

\DeclareMathOperator{\R}{R}

\DeclareMathOperator{\can}{can}

\DeclareMathOperator{\Isoc}{Isoc}
\DeclareMathOperator{\Cris}{Cris}
\DeclareMathOperator{\D}{\mathscr{D}}
\DeclareMathOperator{\sP}{sp}

\DeclareMathOperator{\Sp}{Sp}
\DeclareMathOperator{\Con}{Con}

\begin{document}

\selectlanguage{british}

\thispagestyle{empty}

\title{Comparison between Rigid and Overconvergent Cohomology with Coefficients}
\author{Veronika Ertl}
\date{}
\maketitle

\begin{abstract}
{\noindent
For a smooth scheme over a perfect field of characteristic $p>0$, we generalise a definition of Bloch and introduce overconvergent de Rham-Witt connections. This provides a tool to extend the comparison morphisms of Davis, Langer and Zink between overconvergent de Rham-Witt cohomology and Monsky-Washnitzer respectively rigid cohomology to coefficients. }
\end{abstract}
\vspace{-9pt}
\selectlanguage{french}
\begin{abstract}
{\noindent
Pour un sch\'ema lisse sur un corps parfait de caract\'eristique $p>0$, on g\'en\'eralise une d\'efinition de Bloch et introduit des connexions surconvergentes de de Rham-Witt. Cela fournit une possibilit\'e d'\'etendre \`a des coefficients convenable les morphismes de comparaison de Davis, Langer et Zink entre la cohomologie surconvergente de de Rham-Witt et la cohomologie de Monsky-Washnitzer respectivement la cohomologie rigide. 
}
\end{abstract}

\selectlanguage{british}

\tableofcontents

\let\thefootnote\relax

\newpage

\addcontentsline{toc}{section}{Introduction}

\section*{Introduction}

Let $X$ be a smooth variety over a perfect field $k$ of positive characteristic $p$. In \cite{Bloch3} Spencer Bloch studies connections on the de Rham--Witt complex of $X$ and establishes the following equivalence of categories.

\begin{theo*}[Bloch 2001, \protect{\cite[Thm. 1.1]{Bloch3}}]  The functor $\EE\mapsto(E,\nabla)$ defines an equivalence of category between locally free crystals on $X$ and locally free $W_X$-modules with quasi-nilpotent integrable connections. \end{theo*}

In a corollary he states that this implies also an equivalence of categories between locally free $F$-crystals on $x$ and locally free $W_X$-modules with quasi-nilpotent integrable connection and a Frobenius structure.

The objects considered in this statement form a suitable family of coefficients for crystalline cohomology. Yet, we have to bear in mind, that crystalline cohomology is only a good integral model for rigid cohomology in the case of a proper variety $X$. The appropriate cohomology theory for non-proper varieties was constructed by Davis, Langer and Zink in \cite{DavisLangerZink} in form of the overconvergent de Rham--Witt complex. In an attempt to define coefficients for this overconvergent integral cohomology theory, we realise that Bloch's equivalence of categories provides an interesting approach to obtain overconvergent crystals. 

We define a subcategory of the de Rham--Wit connections from above's theorem, by taking modules over $W^\dagger(\OO_X)$ -- the overconvergent Witt vectors -- together with connections that take values in the overconvergent de Rham--Witt complex. It turns out, that this is a reasonable category of coefficients for the overconvergent cohomology theory. In particular, using this definition we are able to generalise the comparison theorems of Davis, Langer and Zink between overconvergent de Rham--Witt cohomology and Monsky--Washnitzer, respectively rigid cohomology, to coefficients. 

In Section \ref{Sec1} we start out by recalling some facts about crystals and their relation to de Rham--Witt connections as studied by Etesse and Bloch. In Section \ref{Subsec1.3} we define  overconvergent de Rham--Witt connections. 

Section \ref{Sec2} is devoted to the comparison between Monsky--Washnitzer and overconvergent cohomology with coefficients. Thence, we recall Monsky--Washnitzer algebras and define overconvergent connections on modules over them. We follow ideas of Davis, Langer and Zink in Section \ref{Subsec2.2} to compare Monsky--Washnitzer cohomology to overconvergent cohomology using the mentioned connections on the Monsky--Washnitzer side and overconvergent de Rham--Witt connections on the overconvergent side. 

This is also an important ingredient for the comparison between rigid cohomology and overconvergent cohomology with coefficients to which we turn in Section \ref{Sec3}. We recall Berthelot's notion of overconvergent isocrystals in Section \ref{Subsec3.1}. In Sections \ref{Subsec3.2} and \ref{Subsec3.3} we verify that the methods used by Davis, Langer and Zink for the comparison theorem without coefficients, carry over to the case of coefficients. Thus we obtain the following  result.
\begin{theo*}[Theorem \ref{Theo3.2.7}] Let $X$ be a smooth quasi-projective scheme over $k$, and $\edg\in\Isoc^\dagger(X/W(k)$ an isocrystal. Then there is a quasi-isomorphism
$$R\Gamma_{\text{rig}}(X,\edg)\rightarrow R\Gamma\left(X,\edg\otimes W^\dagger\Omega^{_{^\bullet}}_{X/k,\QQ} \right).$$
\end{theo*}

\subsection*{Acknowledgments}

I had the chance to talk to Christopher Davis, Kiran Kedlaya, Lance Miller, Wies\l{}awa Nizio\l{} and Liang Xiao about questions related to this article and I want to thank them here. Part of this work was done at the \'Ecole Normale Sup\'erieure in Lyon, where I was invited by Wies\l{}awa Nizio\l{}, and at the University of Utah, and I am very grateful for the support and hospitality of both institutions.

\section{Crystals and de Rham--Witt connections}\label{Sec1}

Let $X$ be a smooth variety over a perfect field $k$ of positive characteristic $p$. As usual denote by $W(k)$ the ring of $p$-typical Witt vectors of $k$, $W_n(k)$ the $n$-truncated Witt vectors, and throughout the paper, let $K$ be the fraction field of $W(k)$. A classical result in crystalline cohomology is the comparison with the hypercohomology of the de Rham--Witt complex. 

\begin{theo}[Illusie, \protect{\cite[II.1.4]{Illusie}}] For each $n\in \NN_0$ there is a canonical isomorphism between the crystalline cohomology $H{^{_{^\bullet}}}(X\slash W_n(k))$ and the hypercohomology of the de Rham--Witt complex $\HH^{_{^\bullet}}(X,W_n\Omega_{X/k}^{_{^\bullet}})$ which is compatible with the Frobenius action on both sides. More generally, there is a quasi-isomorphism 
$$\R u_\ast\OO_{X/W_n(k)}\xrightarrow{\sim} W_n\Omega_{X/k}^{_{^\bullet}}$$
of sheaves of $W_n(k)$-modules which is functorial in $X/k$. Here $u:(X/W_n(k))_{\text{cris}}\rightarrow X_{\text{Zar}}$ is the natural projection of topoi.\end{theo}

This induces an isomorphism
$$H{^{_{^\bullet}}}(X/W(k))\cong \HH^{_{^\bullet}}(X,W\Omega_{X/k}^{_{^\bullet}}).$$

It is natural to generalise this to crystals.

\subsection{From crystals to connections}

It is well-known that crystals and connections are related.

\begin{theo}[Berthelot-Ogus, \protect{\cite[Theo. 6.6]{BerthelotOgus}}] Let $S$ be a PD-scheme and $X\hookrightarrow Y$ a closed immersion of $S$-schemes with $Y/S$ smooth, then the following categories are equivalent:
\begin{enumerate}\item The category of crystals of $\OO_{X/S}$-modules on $\Cris(X/S)$.
\item The category of $\D_X(Y)$-modules together with a hyper PD-stratification compatible with the canonical one of $\D_X(Y)$.
\item The category of $\D_X(Y)$-modules together with an integrable, quasi-nilpotent connection, compatible with the canonical one on $\D_X(Y)$.\end{enumerate}
\end{theo}
Here $\D_X(Y)$ denotes the structure sheaf of the PD-envelope of $X$ in $Y$. Its canonical hyper PD-stratification is a natural isomorphism
$$\varepsilon: \D_Y(Y\times_S Y)\otimes \D_X(Y)\xrightarrow{\sim} \D_X(Y)\otimes \D_Y(Y\times_S Y)$$
(see \cite[Def. 4.3H]{BerthelotOgus}) where $\D_Y(Y\times_S Y)$ is the structure sheaf of the first infinitesimal neighbourhood of the diagonal. This in turn  corresponds to the mentioned canonical integrable, quasi-nilpotent (in the sense of \cite[Def. 4.10 and 4.14]{BerthelotOgus}) connection
$$\nabla: \D_X(Y)\rightarrow \D_X(Y)\otimes \Omega^1_{Y/S}.$$

Now we get back to our original set-up. Bloch explains how to use the above result with $S=W(k)$ in order to associate to a crystal $\EE$ on $X/W$ a de Rham--Witt connection. One obtains a pro-system of Zariski sheaves on $X$ by setting for each $n\in\NN$ 
$$E_n=E_{W_n(\OO_X)}: U \mapsto \EE(U,\Spec(\Gamma(U,W_n(\OO_X))), \can).$$
To obtain the additional structure of a connection, let $(U,V,\gamma)$ be an object of the crystalline site of $X$. By definition, there is a diagram of PD-morphisms
$$\xymatrix{U \ar@{^{(}->}[r] \ar@2{-}[d] & D_V(V\times V) \ar@<.5ex>[d]^{p_2} \ar@<-.5ex>[d]_{p_1} \\
U \ar@{^{(}->}[r] & V}$$
and by the crystalline property this gives a hyper PD-stratification
$$p_2^\ast(\EE_V)\cong \EE_{D_V(V\times V)} \cong p_1^\ast(\EE_V).$$
The difference of these two morphisms then provides an integrable connection
$$\nabla: \EE_V\rightarrow \EE_V\otimes \Omega^1_{V/W_n(k),\gamma},$$
where $\Omega^{_{^\bullet}}_{V/W_n(k),\gamma}$ is the PD-de Rham complex from \cite[Lem. II.1.1.4]{Etesse}. This induces a connection $\nabla_n: E_n\rightarrow E_n\otimes \Omega^1_{W_n(\OO_X),\can}$, and passing to the quotient gives a unique differential graded module with underlying graded  $W_n\Omega^{_{^\bullet}}_{X/k}$-module $E_n\otimes_{W_n(\OO_X)}W_n\Omega^{_{^\bullet}}_{X/k}$ where the differential is given by
$$E_n\xrightarrow{\nabla}E_n\otimes \Omega^1_{W_n(\OO_X),\can}\rightarrow E_n\otimes W_n\Omega_{X/k}^1.$$

\begin{defi} 
The complex $E_{n}\otimes W_n\Omega^{_{^\bullet}}_{X/k}$ is called the $n$-truncated de Rham--Witt complex with coefficients in the crystal $\EE$. The induced pro-complex 
$$E_{^{_\bullet}}\otimes W_{^{_\bullet}}\Omega^{_{^\bullet}}_{X/k}$$
is the de Rham--Witt pro-complex with coefficients in the crystal $\EE$. Taking limits, we obtain the de Rham--Witt complex with coefficients in the crystal $\EE$.
\end{defi}

Etesse shows in \cite[Prop. II.1.2.7]{Etesse} that if $\EE$ is locally free of finite type, this limit is compatible with the tensor product of $E=\varprojlim E_n$ and $W\Omega^{_{^\bullet}}_{X/k}$ in the sense that
$$E\otimes_{W(\OO_X)}W\Omega_{X/k}^i\xrightarrow{\sim} \varprojlim\left(E_n\otimes_{W_n(\OO_X)}W_n\Omega_{X/k}^i\right)$$
is an isomorphism.
His main result is the following.

\begin{theo}[Etesse 1988, \protect{\cite[Thm. II.2.1]{Etesse}}] Let $u_n:(X/W_n(k))_{\text{cris}}\rightarrow X_{\text{zar}}$ be the canonical projection of topoi and $\EE$ a locally free crystal on $X/W_n(k)$. Then the morphism
$$\R u_n(\EE)\rightarrow E_n\otimes W_n\Omega_{X/k}^{_{^\bullet}}$$
is an isomorphism. \end{theo}

If $X/k$ is in addition proper and $\EE$ of finite type, the previous theorem induces an isomorphism
$$H_{\text{cris}}^{_{^\bullet}}(X/W(k),\EE)\cong\HH^{_{^\bullet}}(X,E\otimes W\Omega^{_{^\bullet}}_{X/k}).$$

Bloch's discussion of the functor $\EE\mapsto(E,\nabla)$ establishes an equivalence of categories between locally free crystals on $X/W(k)$ and locally free $W_X$-modules with a quasi-nilpotent integrable connection. The property of being quasi-nilpotent is defined to be quasi-nilpotent in the sense of \cite[Def. 4.10 and 4.14]{BerthelotOgus} modulo $p$. In the next section, we will introduce a sub-category of the latter category. 

\subsection{Overconvergent de Rham--Witt connections}\label{Subsec1.3}

In \cite{DavisLangerZink2} C. Davis, A. Langer and T. Zink define for a finitely generated $k$-algebra $\overline{A}$ the ring of overconvergent Witt vectors $W^\dagger(\overline{A})\subset W(\overline{A})$ as subring of the usual Witt vectors. In the subsequent paper \cite{DavisLangerZink} they extend this to the definition of an overconvergent de Rham--Witt complex $W^\dagger\Omega^{_{^\bullet}}_{\overline{A}/k}\subset W\Omega^{_{^\bullet}}_{\overline{A}/k}$ which can in fact be globalised to a sheaf on a smooth variety $X$ over $k$, and so can of course $W^\dagger$. 

Let $E$ be a $W^\dagger(\OO_X)$-module on $X$. 

\begin{defi} 
An overconvergent de Rham--Witt connection on $E$ will be a map
$$\nabla: E\rightarrow E\otimes_{W^\dagger}W^\dagger\Omega^1_{X/k}$$
which satisfies the Leibniz rule
$$\nabla(\omega e)=\omega\nabla(e) + e\otimes d\omega.$$
\end{defi}

This definition makes sense, as by construction, the differential of the de Rham--Witt complex takes overconvergent elements to overconvergent elements. 
\begin{defi}
As usual, we say that an overconvergent connection is integrable if
$$\nabla^2=0.$$
It is quasi-nilpotent if it is quasi-nilpotent modulo $p$ in the sense of \cite[Def. 4.10 and 4.14]{BerthelotOgus}.
\end{defi}
A morphism 
$$\psi:(E_1,\nabla_{E_1})\rightarrow (E_2,\nabla_{E_2})$$
of $W^\dagger(\OO_X)$-modules with integrable connections is a morphism of $W^\dagger(\OO_X)$-modules compatible on each side with the connection, or in other words, it is a morphism of complexes
$$E_1\otimes W^\dagger\Omega^{_{^\bullet}}_{X/k}\rightarrow E_2\otimes W^\dagger\Omega^{_{^\bullet}}_{X/k}$$
where the boundary maps are given by the connections. 

We denote the category of locally free quasi-nilpotent integrable de Rham--Witt connections by $\Con(X/W(k))$ and the faithful subcategory of locally free quasi-nilpotent integrable overconvergent de Rham--Witt connections by $\Con^\dagger(X/W(k))$.

\begin{rem}\label{rem132} It is not clear that the category of overconvergent de Rham--Witt connections forms a full subcategory of the category of convergent de Rham--Witt connections. However, analogous to isocrystals, where one defines $F$-isocrystals, one can define $F$-de Rham--Witt connections, denoted by $F$-$\Con(X/W(k))$, which correspond to $F$-crystals as defined in \cite{Illusie} and overconvergent $F$-de Rham--Witt connections denoted by $F$-$\Con^\dagger(X/W(k))$. With this extra structure, we show in \cite{Ertl3} that the forgetful functor
$$j^\ast: \f\text{-}\Con^\dagger(X/W(k))\rightarrow\f\text{-}\Con(X/W(k))$$
is fully faithful.\end{rem}

We restrict our attention now to locally free $W^\dagger(\OO_X)$-modules with overconvergent integrable quasi-nilpotent connections $(E,\nabla)$. 

\begin{rem}In light of Remark \ref{rem132} and taking into account that Bloch established an equivalence of categories between $\Con(X/W)$ and the category of locally free crystals on $X$, one would like to have a suitable notion of overconvergent crystals on $X$.\end{rem}

\section{Overconvergent de Rham--Witt and Monsky--Washnitzer cohomology with coefficients}\label{Sec2}

Let $X/k$ be as before. The goal of this section is to extend the comparison morphism of Davis, Langer and Zink between overconvergent de Rham--Witt and Monsky--Washnitzer algebra to coefficients. This lies the ground work for the comparison between overconvergent de Rham--Witt and rigid cohomology.

\subsection{Monsky--Washnitzer algebras and overconvergent connections}\label{Subsec2.1}

For a complete discrete valuation ring $R$ and a quotient of a polynomial algebra $B=R[t_1,\ldots,t_n]/(f_1,\ldots,f_m)$ the associated Monsky--Washnitzer algebra is according to \cite{vanderPut}
$$B^\dagger=R\langle t_1,\ldots,t_n\rangle^\dagger/(f_1,\ldots,f_m),$$
where $R\langle t_1,\ldots,t_n\rangle^\dagger$ consists of those power series $\sum_{\alpha}a_\alpha t^\alpha$ in $R \llbracket t_1,\ldots,t_n\rrbracket$ such that for some $C>0$ and $0<\rho<1$ 
$$|a_\alpha|\leqslant C\rho^{|\alpha|}$$
for all coefficients. The elements are called overconvergent power series and can be characterised by the fact that they are converging in a poly-disc with radii $\rho_i>1$. 

The algebra $R\langle t_1,\ldots,t_n\rangle^\dagger$ satisfies Weierstra\ss{} preparation and division and has the Artin approximation property. 

We turn now to the special case, when $R=W(k)$ for the field $k$ above. The notation is as in \cite{Davis}. Let $\Spec(\overline{B})$ be a smooth affine variety over $k$. There exists a smooth lift of the form $B=W(k)[x_1,\ldots,x_n]/(f_1,\ldots,f_m)$ (cf.  \cite[p. 35]{vanderPut}) and we consider the associated Monsky--Washnitzer algebra
$$B^\dagger:=W(k)\langle x_1,\ldots,x_n\rangle^\dagger/(f_1,\ldots,f_m).$$

Let $\Omega^{_{^\bullet}}_{B^\dagger/W(k)}$ be the module of continuous differentials of $B^\dagger$ over $W(k)$. The Monsky--Washnitzer cohomology of $\Spec \overline{B}$ is then calculated by the rational complex $\Omega^{_{^\bullet}}_{B^\dagger/W(k),\QQ}:=\Omega^{_{^\bullet}}_{B^\dagger/W(k)}\otimes\QQ$
$$H_{\text{MW}}^i(\overline{B}/K)=\HH^i(\Omega^{_{^\bullet}}_{B^\dagger/W(k),\QQ}).$$
These notions are well-defined, independent of the choice of representation, and functorial.

To introduce coefficients to this cohomology theory, in the spirit of overconvergent crystals, we make the following definition.

\begin{defi}\label{MWIntCon} 
Let $E_{B^\dagger}$ be a free $B^\dagger$-module. A connection on $E_{B^\dagger}$ is a map
$$\nabla_{B^\dagger}:E_{B^\dagger}\rightarrow E_{B^\dagger}\otimes_{B^\dagger}\Omega^1_{B^\dagger/W(k)}$$
satisfying the Leibniz rule $\nabla_{B^\dagger}(we)=w\nabla_{B^\dagger}(e)+e\otimes dw$. It is said to be integrable if $\nabla_{B^\dagger}^2=0$ as usual. It is called quasi-nilpotent, if it is quasi-nilpotent modulo $p$ in the sense of \cite[Def. 4.10 and 4.14]{BerthelotOgus}.
\end{defi}
The Monsky--Washnitzer cohomology of $\Spec\overline{B}$ with coefficients in $E_{B^\dagger}$ is then
$$H_{\text{MW}}^i(\overline{B}/K,E_{B^\dagger})=\HH^i(E_{B^\dagger}\otimes\Omega^{_{^\bullet}}_{B^\dagger/W(k),\QQ}).$$

\subsection{Comparison with overconvergent cohomology}\label{Subsec2.2}

Let $W^\dagger\Omega^{_{^\bullet}}_{\overline{B}/k}$ be the overconvergent subcomplex of the de Rham--Witt complex as defined in \cite[Chapter 3]{Davis} or \cite[Section 3]{DavisLangerZink}. For a smooth lift of the Frobenius $F:B^\dagger\rightarrow B^\dagger$, which always exists, there is a monomorphism
$$t_F:B^\dagger\rightarrow W(\overline{B})$$
which has in fact image in the overconvergent subring $W^\dagger(\overline{B})\subset W(\overline{B})$ and induces by the universal property of K\"a{}hler differentials and functoriality a comparison map
$$t_F:\Omega^{_{^\bullet}}_{B^\dagger/W(k)}\rightarrow W^\dagger\Omega^{_{^\bullet}}_{\overline{B}/k}.$$
Although this depends a priori on the choice of Frobenius lift, Davis proves in \cite[Cor. 4.1.11]{Davis} (see also \cite[Prop. 3.26]{DavisLangerZink} that for standard \'etale affines it is independent in the derived category. 

The main result of Davis, Langer and Zink \cite[Prop. 3.23]{DavisLangerZink} regarding this comparison morphism is that for an arbitrary smooth algebra $B$ the kernel and cokernel of the induced homomorphism on cohomology is annihilated by $p^{2\kappa}$, where $\kappa=\lfloor\log_p(\dim \overline{B})\rfloor$. As a corollary they obtain the following result.

\begin{cor}[Davis--Langer--Zink 2011 \protect{\cite{DavisLangerZink}}]\renewcommand{\labelenumi}{(\alph{enumi})} \begin{enumerate}\item Let $\dim \overline{B}<p$. Then the induced map on cohomology $t_{F\ast}$ is an isomorphism.
\item In general, there is a rational isomorphism
$$H_{\text{MW}}^\ast(\overline{B}/K)\cong\HH^\ast(W^\dagger\Omega^{_{^\bullet}}_{\overline{B}/k,\QQ}),$$
between Monsky--Washnitzer cohomology and rational overconvergent de Rham--Witt cohomology.\end{enumerate}\end{cor}

We want to generalise this to a comparison between Monsky--Washnitzer cohomology and overconvergent cohomology with coefficients. For this let $(E_{B^\dagger},\nabla_{B^\dagger})$ be a free $B^\dagger$-module of rank $r$ with integrable connection as defined in (\ref{MWIntCon}). Let $\{e_1,\ldots, e_r\}$ be a set of generators. Then using the map $t_F: B^\dagger\rightarrow W^\dagger(\overline{B})$ consider the tensor product
$$E_F:=t^\ast_F(E_{B^\dagger})= E_{B^\dagger}\otimes_{t_F}W^\dagger(\overline{B}).$$
This is a free $W^\dagger(\overline{B})$-module of rank $r$ where by abuse of notation, we denote the induced set of generators also by $\{e_1,\ldots, e_r\}$.  The connection $\nabla_{B^\dagger}$ gives via the inverse image functor rise to a map 
$$\nabla_F:=t_F^\ast\nabla_{B^\dagger}: E_F\rightarrow E_F\otimes_{W^\dagger(\overline{B})}W^\dagger\Omega^{_{^\bullet}}_{\overline{B}}.$$
It fits into a commutative diagram
\begin{equation}\label{diagNabla}
\xymatrix{E_{B^\dagger}\ar[r] \ar[d]_{\nabla_{B^\dagger}} & E_F \ar[d]^{\nabla_F} \\
E_{B^\dagger}\otimes_{B^\dagger}\Omega^1_{B^\dagger} \ar[r] & E_F\otimes_{W^\dagger(\overline{B})}W^\dagger\Omega_{\overline{B}/k}^1}
\end{equation}
where the horizontal maps induced by $t_F$ take basis to basis. 
 
\begin{prop} 
The above constructed map $\nabla_F:E_F\rightarrow E_F\otimes W^\dagger\Omega_{\overline{B}/k}^1$ is an integrable connection.
\end{prop}
\begin{proof} 
As the original connection $\nabla_{B^\dagger}$ satisfies the Leibniz rule, it is clear by the definition and the commutativity of diagram (\ref{diagNabla}), that $\nabla_F$ satisfies the Leibniz rule as well. Indeed, it is easy to check this for an element $e_i\otimes w$ with $e_i$ a generator of $E_F$, and the general case follows right away. 

Since $t_F$ is compatible with the differentials on either side, we might extend the diagram to higher differentials and it still stays commutative. Again integrability follows from the fact that, $\nabla_{B^\dagger}$ and $d$ are integrable connections themselves and therefore for $e_i\in E_F$ a generator $\nabla^2_F(e_i)=0$. With $w\in W^\dagger(\overline{B})$ we calculate 
\begin{eqnarray*}\nabla_F^2\left(e_i\otimes w\right) &=& \nabla_F\left(\nabla_F (e_i)\otimes w+ e_i\otimes dw\right)\\
&=& \nabla_F^2( e_i)\otimes w-\nabla_F e_i\otimes dw+\nabla_F e_i\otimes  dw + e_i\otimes d^2w = 0\end{eqnarray*}
with the correct sign convention and the general case follows by linearity.
\end{proof}

\begin{lem} 
With the notation as before, if $\nabla_{B^\dagger}$ is quasi-nilpotent, so is $\nabla_F$.
\end{lem}
\begin{proof} 
Quasi-nilpotence is a local property which can be checked in local coordinates since by \cite[Remark 4.11]{BerthelotOgus} it is independent of the choice of coordinate system. For both, $\nabla_{B^\dagger}$ and $\nabla_F$ it is defined modulo $p$. But modulo $p$ these two connections coincide and the statement is clear.
\end{proof}

For different choices of Frobenius lifts on $B^\dagger$, we obtain different $W^\dagger(\overline{B})$-modules $E_F$, which are  isomorphic as free modules since $t_F$ is injective. Additionally, the comparison map
$$t_F:E_{B^\dagger}\otimes_{B^\dagger}\Omega^{_{^\bullet}}_{B^\dagger/W(k)}\rightarrow E_F\otimes_{W^\dagger(\overline{B})}W^\dagger\Omega^{_{^\bullet}}_{\overline{B}/k}$$
still depends on the choice of Frobenius lift $F$ and it is not clear a priori that two different choices of Frobenius induce the same map on cohomology. 

Davis, Langer and Zink show however that for $\overline{B}$ standard \'etale affine (for the definition see \cite[Def. 4.1.3]{Davis}) different choices of $F$ induce chain homotopic maps on complexes $t_F: \Omega^{_{^\bullet}}_{B^\dagger/W(k)}\rightarrow W^\dagger\Omega^{_{^\bullet}}_{\overline{B}/k}$ \cite[Prop. 3.27]{DavisLangerZink}. They use this to deduce from the standard \'etale affine case, that the map $t_F$  is a rational quasi-isomorphism for general smooth $k$-algebras $\overline{B}$.

We will extend their results to overconvergent connections, starting as they did with the case of $\overline{B}$ being standard \'etale affine. 

\begin{prop}\label{Prop224}
Let $\overline{B}$ be standard \'etale affine and $(E_{B^\dagger},\nabla_{B^\dagger})$ a free integrable connection as introduced above. Then  the comparison map $t_F:E_{B^\dagger}\otimes_{B^\dagger}\Omega^{_{^\bullet}}_{B^\dagger/W(k)}\rightarrow E_F\otimes_{W^\dagger(\overline{B})}W^\dagger\Omega^{_{^\bullet}}_{\overline{B}/k}$ is a quasi-isomorphism.
\end{prop}
\begin{proof}
In the present case of a standard \'etale affine algebra the complex $W^\dagger\Omega^{_{^\bullet}}_{\overline{B}/k}$ has a decomposition in terms of integral and rational weights of the basic Witt differentials
$$W^\dagger\Omega^{_{^\bullet}}_{\overline{B}/k} = W^{\dagger,\text{int}}\Omega^{_{^\bullet}}_{\overline{B}/k}\oplus W^{\dagger,\text{frac}}\Omega^{_{^\bullet}}_{\overline{B}/k}$$
which is compatible with the differential graded structure. Davis, Langer and Zink show in the proof of \cite[Theo. 3.19]{DavisLangerZink} that on the one hand $t_F$ maps in fact into the integral part of the overconvergent  complex and moreover $t_F:\Omega^{_{^\bullet}}_{B^\dagger/W(k)}\rightarrow W^{\dagger,\text{int}}\Omega^{_{^\bullet}}_{\overline{B}/k}$ is an isomorphism, and on the other hand, that $W^{\dagger,\text{frac}}\Omega^{_{^\bullet}}_{\overline{B}/k}$ is acyclic. 

We wish to apply the same strategy but have to be a bit careful tensoring with $E_{B^\dagger}$. The map of complexes $t_F$ can be represented by a commutative diagram of which we provide the first square
$$\xymatrix{  E_{B^\dagger}    \ar[rr]^{t_F\qquad\qquad\qquad} \ar[dd]_{\nabla_{B^\dagger}}   &    &         E_F\otimes W^{\dagger,\text{int}}(\overline{B})\oplus E_F\otimes W^{\dagger,\text{frac}}(\overline{B})       \ar[dd]^{\nabla_F\big|_{\text{int}} +\nabla_F\big|_{\text{frac}} }\\
&&\\
E_{B^\dagger}\otimes\Omega_{B^\dagger/W(k)}^1       \ar[d]_{\nabla_{B^\dagger}} \ar[rr]^{t_F\qquad\qquad\qquad}    & &    E_F\otimes W^{\dagger,\text{int}}\Omega^1_{\overline{B}/k} \oplus E_F \otimes W^{\dagger,\text{frac}}\Omega^1_{\overline{B}/k} \ar[d]^{\nabla_F\big|_{\text{int}} +\nabla_F\big|_{\text{frac}} }\\
&&}$$
where the notation $\big|_{\text{int}}$ and $\big|_{\text{frac}}$ means restriction of the coefficients to the integral and fractional part respectively. Recall that the image of $t_F$ in $W^{\dagger,\text{frac}}\Omega^{_{^\bullet}}_{\overline{B}/k}$, which is acyclic, is trivial. Thus the tensor product $E_F\otimes W^{\dagger,\text{frac}}\Omega^{_{^\bullet}}_{\overline{B}/k}= E_{B^\dagger}\otimes_{t_F} W^{\dagger,\text{frac}}\Omega^{_{^\bullet}}_{\overline{B}/k}$ is trivial. 

From the rank one case it follows then that $t_F:E_{B^\dagger}\otimes\Omega^{_{^\bullet}}_{B^\dagger/W(k)}\rightarrow E_F\otimes W^{\dagger,\text{int}}\Omega^{_{^\bullet}}_{\overline{B}}$ is an isomorphism and the claim is shown.\end{proof}
 
It remains to generalise this result to an arbitrary smooth algebra $\overline{B}$. Consider an overconvergent Witt lift 
$$t_F:B^\dagger\rightarrow W^\dagger(\overline{B})$$
uniquely determined by a Frobenius lift $F$ to $B^\dagger$. Let $t_F$ as before also denote  the induced map of complexes which for a free $B^\dagger$-module with integrable connection $(E_{B^\dagger},\nabla_{B^\dagger})$ gives a map 
$$t_F:E_{B^\dagger}\otimes\Omega^{_{^\bullet}}_{B^\dagger}\rightarrow E_F\otimes  W^\dagger\Omega^{_{^\bullet}}_{\overline{B}}.$$
Passing to cohomology, we generalise the result of \cite[Cor. 3.25]{DavisLangerZink}.

\begin{prop}\label{Prop225} 
The morphism $t_F$ induces a rational isomorphism 
$$H_{\text{MW}}{^{_{^\bullet}}}(A/K,E_{B^\dagger})\cong\HH{^{_{^\bullet}}}(E_F\otimes W^\dagger\Omega^{_{^\bullet}}_{\overline{B}/k,\QQ})$$
between Monsky--Washnitzer cohomology and overconvergent cohomology with coefficients.
\end{prop}

\begin{proof} 
Following \cite{DavisLangerZink}, we note first that the map $t_F:E_{B^\dagger}\otimes\Omega^{_{^\bullet}}_{B^\dagger}\rightarrow E_F\otimes W^\dagger\Omega^{_{^\bullet}}_{\overline{B}}$ induces a map of complexes of Zariski sheaves on $\Spec\overline{B}$:
$$\widetilde{t}_F:\widetilde{E}_{B^\dagger}\otimes\widetilde{\Omega}^{_{^\bullet}}_{B^\dagger/W(k)}\rightarrow \widetilde{E}_F\otimes W^\dagger\Omega^{_{^\bullet}}_{\Spec\overline{B}/k}.$$
Recall from \cite[Proposition 1.2]{DavisLangerZink} and \cite[Lemma 7]{Meredith} that for Zariski topology 
$$H^i(\Spec\overline{B},\widetilde{\Omega}^d_{B^\dagger/W(k)})=0=H^i(\Spec\overline{B},W^\dagger\Omega^d_{\Spec\overline{B}/k}).$$ 
Since $E_{B^\dagger}$ is a free $B^\dagger$-module and $E_F$ is a free $W^\dagger(\overline{B})$-module, the same equalities hold true with coefficients. Thus we have
\begin{eqnarray*} R\Gamma(\Spec\overline{B},\widetilde{E}_{B^\dagger}\otimes\widetilde{\Omega}^{_{^\bullet}}_{B^\dagger/W(k)}) &=& E_{B^\dagger}\otimes\Omega^{_{^\bullet}}_{B^\dagger/W(k)}\\
R\Gamma(\Spec\overline{B},\widetilde{E}_F\otimes W^\dagger\Omega^{_{^\bullet}}_{\Spec \overline{B}/k}) &=& E_F\otimes W^\dagger\Omega^{_{^\bullet}}_{\overline{B}/k}\end{eqnarray*}
which means, as point out Davis, Langer and Zink, that we can reconstruct $t_F$ from $\widetilde{t}_F$. This argument holds verbatim for the induced maps $t_F$ and $\widetilde{t}_F$ on the rational complexes (where $p$ has been inverted). 

Let $\{U_i=\Spec\overline{B}_i\}_{i\in I}$ be a finite cover of $\Spec\overline{B}$ made out of standard \'etale affines, such that their intersections are again standard \'etale affine, which always exists by a result of Kedlaya in \cite{Kedlaya3}. Let $t_{F,i}: \Omega^{_{^\bullet}}_{B^\dagger_i/W(k)}\rightarrow W^\dagger\Omega^{_{^\bullet}}_{\overline{B}_i/k}$ be the localisation of $t_F$ to $\overline{B}_i$ which induces the localisation 
$$t_{F,i}:E_{B^\dagger}\otimes\Omega^{_{^\bullet}}_{B^\dagger_i/W(k)}\rightarrow E_F\otimes W^\dagger\Omega^{_{^\bullet}}_{\overline{B}_i/k}.$$

Davis, Langer and Zink show, that the map $p^\kappa t_{F,i}: \Omega^{_{^\bullet}}_{B^\dagger_i/W(k)}\rightarrow W^\dagger\Omega^{_{^\bullet}}_{\overline{B}_i/k}$ on the standard \'etale affine $\overline{B}_i$ is homotopic to a map $p^\kappa t_{F'}$ where $t_{F'}$ is determined uniquely on $\overline{B}_i$ via a lifting $F'$ of Frobenius to $\widetilde{B}_i$, and is a quasi-isomorphism. Therefore the rational map
$$t_{F,i}:\Omega^{_{^\bullet}}_{B^\dagger_i/W(k),\QQ}\rightarrow W^\dagger\Omega^{_{^\bullet}}_{\overline{B}_i/k,\QQ}$$
is a quasi-isomorphism. As $W^{\dagger,\text{frac}}\Omega^{_{^\bullet}}_{\overline{B}_i/k,\QQ}$ is acyclic, and the projection 
$$W^\dagger\Omega^{_{^\bullet}}_{\overline{B}_i/k,\QQ}=W^{\dagger,\text{int}}\Omega^{_{^\bullet}}_{\overline{B}_i/k,\QQ}\oplus W^{\dagger,\text{frac}}\Omega^{_{^\bullet}}_{\overline{B}_i/k,\QQ}\rightarrow W^{\dagger,\text{int}}\Omega^{_{^\bullet}}_{\overline{B}_i/k,\QQ}$$
is a quasi-isomorphism, $t_{F,i}$ factors through $W^{\dagger,\text{int}}\Omega^{_{^\bullet}}_{\overline{B}_i/k,\QQ}$ in the derived category. 

Consequently, if we tensor these rational complexes with the free module $E_{B^\dagger}$ we obtain a diagram similar to the one above
$$\xymatrix{  E_{B_i^\dagger,\QQ}    \ar[rr]^{t_{F,i}\qquad\qquad\qquad} \ar[dd]_{\nabla_{B_i^\dagger}}   &    &         E_F\otimes W^{\dagger,\text{int}}(\overline{B}_i)_{\QQ}\oplus E_F\otimes W^{\dagger,\text{frac}}(\overline{B}_i)_{\QQ}       \ar[dd]^{\nabla_F\big|_{\text{int}} +\nabla_F\big|_{\text{frac}} }\\
&&\\
E_{B_i^\dagger}\otimes\Omega_{B_i^\dagger/W(k),\QQ}^1       \ar[d]_{\nabla_{B_i^\dagger}} \ar[rr]^{t_F\qquad\qquad\qquad}    & &    E_F\otimes W^{\dagger,\text{int}}\Omega^1_{\overline{B}_i/k,\QQ} \oplus E_F \otimes W^{\dagger,\text{frac}}\Omega^1_{\overline{B}_i/k,\QQ} \ar[d]^{\nabla_F\big|_{\text{int}} +\nabla_F\big|_{\text{frac}} }\\
&&}$$
where the complex $E_F\otimes W^{\dagger,\text{frac}}\Omega^{_{^\bullet}}_{\overline{B}_i/k,\QQ}$ is acyclic. Thus again the map
$$t_{F,i}:E_{B_i^\dagger}\otimes\Omega^{_{^\bullet}}_{B^\dagger_i/W(k),\QQ}\rightarrow E_F\otimes W^\dagger\Omega^{_{^\bullet}}_{\overline{B}_i/k,\QQ}$$
is a quasi-isomorphism, and by what we said above, we conclude that 
$$t_{F}:E_{B^\dagger}\otimes\Omega^{_{^\bullet}}_{B^\dagger/W(k),\QQ}\rightarrow E_F\otimes W^\dagger\Omega^{_{^\bullet}}_{\overline{B}/k,\QQ}$$
is a quasi-isomorphism.\end{proof}

The proposition also shows that $\HH^{_{^\bullet}}(E_F\otimes W^\dagger\Omega_{X/k}^{_{^\bullet}}$ is independent of the choice of a Frobenius lift $F$.


\section{Overconvergent de Rham--Witt and rigid cohomology with coefficients}\label{Sec3}

The goal of this section is to globalise the result of the previous one to a quasiprojective smooth variety $X$ over $k$. More precisely, we want to compare rigid cohomology with a category of coefficients to overconvergent de Rham--Witt connections as introduced Section 1. An appropriate category of coefficients for our purposes is the category of overconvergent isocrystals. We will build on Davis, Langer and Zink's comparison map in \cite{DavisLangerZink} 
\begin{equation}\label{DLZCompRig} R\Gamma_{\text{rig}}(X)\rightarrow R\Gamma(X,W^\dagger\Omega_{X/k,\QQ}),\end{equation}
where the index $\QQ$ denotes that $p$ was inverted. 

\subsection{Overconvergent isocrystals}\label{Subsec3.1}

Let $X\subset \overline{X}$ be a compactification which always exists according to Nagata, let $]X[$ and $]\overline{X}[$ be the tube of $X$ and $\overline{X}$, respectively, as defined by Berthelot  \cite[Sec. 1]{Berthelot}, and $j^\dagger\OO_{]\overline{X}[}$ the sheaf of overconvergent analytic functions on $X$, as well as for a strict neighbourhood $V$ of $]X[$ in $]\overline{X}[$, let $j^\dagger\OO_V$ be the sheaf of overconvergent analytic functions on $V$  (cf. \cite[Sec. 1]{Berthelot2}).

\begin{defi}
An overconvergent isocrystal on $(X\subset\overline{X})$ is a coherent $j^\dagger\OO_{]\overline{X}[}$-module with an integrable connection  whose Taylor series converges on a strict neighbourhood of the diagonal. They form a category denoted by $\Isoc^\dagger(X/W(k))$. \end{defi}

According to \cite[Theo. 2.3.5]{Berthelot} this is independent of the choice of the compactification $\overline{X}$ and therefore the notation $\Isoc^\dagger(X/W(k))$ is justified. Note also that an isocrystal in our setting is automatically locally free of finite rank \cite[Prop. 2.2.3]{Berthelot}.

Rigid cohomology with coefficients in an isocrystal $\edg\in \Isoc^{\dagger}(X/W(k))$ can be computed on a strict neighbourhood $V$ of $]X[$ in $]\overline{X}[$
$$R\Gamma_{\text{rig}}(X/K,\edg):=R\Gamma(]\overline{X}[,\edg\otimes j^\dagger \Omega^{_{^\bullet}}_{]\overline{X}]}= R\Gamma(V,j^\dagger(\edg_V\otimes \Omega^{_{^\bullet}}_V)).$$

We can also calculate rigid cohomology using the dagger spaces of Gro\ss{}e-Kl\"onne \cite{GrosseKloenne}. The idea is to associate to a smooth scheme a rigid analytic space endowed with an overconvergent structure sheaf. With the notation from above, the tube $]\overline{X}[$ seen as a  partially proper rigid space is equivalent to a partially proper dagger space $]\overline{X}[^\dagger$ (see \cite[Thm. 2.27]{GrosseKloenne}). Let $]X[^\dagger$ be the open subspace of $]\overline{X}[^\dagger$ whose underlying set can be identified with $] X[$. Then Grosse-Kl\"onne proves in \cite[Thm. 5.1]{GrosseKloenne} that for a coherent $\OO_{]\overline{X}[^\dagger}$-module $\mathscr{F}$ and the associated coherent $\OO_{]\overline{X}[}$-modules $\F'$ there is a canonical isomorphism
$$H^i(]X[^\dagger,\mathscr{F})\cong H^i\left(]\overline{X}[,j^\dagger\mathscr{F}'\right).$$
In particular this is true for the de Rham complex, thus giving a comparison to rigid cohomology $H_{\text{dR}}^i(]X[^\dagger)=H^i_{\text{rig}}(X/K)$.

\begin{lem}  An overconvergent isocrystal $\edg\in\Isoc^\dagger(X/W(k))$ is equivalent to a coherent $\OO_{]\overline{X}[}$-module with integrable connection.\end{lem}
\begin{proof} 
By definition, the overconvergent isocrystal $\edg$ has a realisation in the sense that there exists a good open neighbourhood $V$ of $]X[$ in $]\overline{X}[$ and a coherent module with integrable connection $\edg_V$ on $V$ such that $\edg=j^\dagger\edg_V$ (cf. \cite[Prop. 2.2.3]{Berthelot}). Since $\edg$ is locally free, it follows that there is a coherent $\OO_{]\overline{X}[}$-module $\edg_{]\overline{X}[}$ with integrable connection such that $\edg=j^\dagger\edg_{]\overline{X}[}$ .

According to \cite[Thm. 2.26]{GrosseKloenne} the functor from dagger spaces to rigid spaces defined by Gro\ss{}e-Kl\"onne in \cite[Thm. 2.19]{GrosseKloenne} induces an equivalence between coherent $\OO_{]\overline{X}[}$-modules and coherent $\OO_{]\overline{X}[^\dagger}$-modules.  The claim follows. \end{proof}

\begin{cor}\label{Cor313} Let $\edg\in\Isoc^\dagger(X/W(k))$. Then there is a canonical isomorphism 
$$H^i(]X[^\dagger,\edg_{]X[^\dagger}\otimes\Omega^{_{^\bullet}}_{]X[^\dagger})\cong H^i_{\text{rig}}(X/K,\edg)$$
of cohomology groups with coefficients.\end{cor}

\subsection{Local results}\label{Subsec3.2}

Here we assume that $X$ is smooth and affine and consider an  isocrystal $\edg\in\Isoc^\dagger(X/W(k))$ and let $\nabla$ be the connection that comes with it. We follow closely the proceeding of \cite[Section 4]{DavisLangerZink}. 

Let $X=\Spec\overline{A}$ be smooth and affine over $k$, $F=\Spec B$ a Witt lift together with a morphism $\varkappa:B\rightarrow W^\dagger(\overline{A})$ which lifts $B\rightarrow \overline{A}$ such that $(X,F,\varkappa)$ is an overconvergent Witt frame as defined in \cite[Def. 4.1]{DavisLangerZink}. Denote by $\widehat{F}$ the $p$-adic completion of $F$ and $]X[_{\widehat{F}}$ the tubular neighbourhood of $X$ in the rigid analytic space $\widehat{F}_K$. Furthermore let $V\subset F^{\text{an}}_K$ be a strict neighbourhood of $]X[_{\widehat{F}}$. The rigid cohomology of $X$ is then $R\Gamma_{\text{rig}}(X)=R\Gamma(V,j^\dagger\Omega_V)$ which is independent of the choice of $V$. 

Davis, Langer and Zink construct a map 
\begin{equation}\label{DLZCompRigLoc}\psi_F:\Gamma(V,j^\dagger\Omega_V)\rightarrow W^\dagger\Omega_{\overline{A}/k,\QQ}\end{equation}
for each overconvergent Witt frame $(X,F,\varkappa)$ and each strict neighbourhood $V$ and show that it factors canonically through $R\Gamma(V,j^\dagger\Omega_V)$. More precisely, they construct a morphism in degree zero
$$\psi_F:\Gamma(V,j^\dagger\OO_V)\rightarrow W^\dagger(\overline{A})_{\QQ}$$
and use the universal property of K\"ahler differentials. This enables us to copy the method we used in Section 2, by reason that it makes now sense to consider the object
$$\edg_F:=\edg\otimes_{\psi_F} W^\dagger(\overline{A})_{\QQ}$$
as a locally free $W^\dagger(\overline{A})_{\QQ}$-module and endow it with the connection
$$\nabla_F:=\psi_F^\ast\nabla:\edg_F\rightarrow \edg_F\otimes_{W^\dagger(\overline{A})_{\QQ}}.$$
Then the diagram
\begin{equation}\label{DiagEdg}\xymatrix{\edg\ar[r]\ar[d]_{\nabla}&\edg_F\ar[d]^{\nabla_F}\\
\edg\otimes\Omega^1_V\ar[r]&\edg_F\otimes(W^\dagger\Omega^1_{\overline{A}/k,\QQ})}\end{equation}
where the horizontal maps are induced by $\psi_F$ is commutative.

\begin{lem} The pair $(\edg_F,\nabla_F)$ is a locally free $W^\dagger(\overline{A})_{\QQ}$-module with integrable connection.\end{lem}
\begin{proof} It is clear that $\edg_F$ is locally free, so it remains only to show the Leibniz rule and integrability for the connection $\nabla_F$. The argument is almost identical to the previous section. Since by hypothesis the original connection $\nabla$ satisfies the Leibniz rule, it follows from the construction of $\nabla_F$ and the commutative diagram (\ref{DiagEdg}) that this is also the case for $\nabla_F$. Again, this can easily be checked locally on generators, and the general case follows. 

Because the morphism $\psi_F$ is compatible with the differentials on each side of the diagram (\ref{DiagEdg}), it is possible to extend it to the whole complex. As in the previous section we make use of the fact that $\nabla$ and $d$ are integrable. For a generator $e_i\in\edg$ and $w\in W^\dagger(A)_{\QQ}$ we calculate
\begin{eqnarray*}\nabla_F^2\left(e_i\otimes w\right) &=& \nabla_F\left(\nabla_F (e_i)\otimes w+ e_i \otimes dw\right)\\
&=& \nabla^2_F (e_i)\otimes w-\nabla_F e_i\otimes dw + \nabla_F e_i\otimes dw + e\otimes d^2w = 0.\end{eqnarray*}
Note the sign convention displayed in the last line. The general case follows by linearity. This shows integrability. \end{proof}

This result enables us to extend the comparison morphism (\ref{DLZCompRigLoc}) to coefficients
\begin{equation}\label{CompRigCoeff}\psi_F: \Gamma\left(V,j^\dagger(\edg_V\otimes\Omega_V)\right)\rightarrow \edg_F\otimes W^\dagger\Omega_{\overline{A}/k,\QQ}  \end{equation}
and we show following \cite{DavisLangerZink} that it factors through the derived global section functor.

\begin{lem}\label{LemComp} 
The last morphism factors canonically through a morphism 
$$\psi: R\Gamma\left(V,j^\dagger(\edg_V\otimes\Omega_V)\right)\rightarrow\edg_F\otimes W^\dagger\Omega_{\overline{A}/k,\QQ}.$$
\end{lem}
\begin{proof} 
Fix the strict neighbourhood $V$ of $]X[_{\widehat{F}}$ in $F^{\text{an}}_K$. The argumentation of \cite[(4.29)]{DavisLangerZink} can be transferred to the current situation verbatim. The morphism (\ref{CompRigCoeff}), as well as the original one, is defined using a fundamental system of quasicompact strict neighbourhoods of $]X[_{\widehat{F}}$ in $V$. For an explicit description of these affinoids see \cite{DavisLangerZink}. For a suitable subsystem of strict neighbourhoods $\{V_{\lambda,\eta}\}_{\lambda,\eta}$  we obtain a composition 
\begin{eqnarray*} R\Gamma\left(V,j^\dagger(\edg_V\otimes\Omega_V)\right)&\rightarrow& \varinjlim R\Gamma\left(V_{\lambda,\eta},(\edg_{V_{\lambda,\eta}}\otimes\Omega_{V_{\lambda,\eta}})\right)\cong\\
&& \cong \varinjlim \Gamma\left(V_{\lambda,\eta},(\edg_{V_{\lambda,\eta}}\otimes\Omega_{V_{\lambda,\eta}})\right)\rightarrow \edg_F\otimes W^\dagger\Omega_{\overline{A}/k,\QQ},\end{eqnarray*}
where the first morphism is induced by restriction. 
\end{proof}

As a consequence, we have defined for each overconvergent Witt frame $(X,F,\varkappa)$ a morphism
\begin{equation}\label{CompEdgFrames} R\Gamma_{\text{rig}}(X,\edg)\rightarrow\edg_F\otimes W^\dagger\Omega_{\overline{A}/k,\QQ}.\end{equation}
A priori, this depends by the very nature of its construction on the choice of a Witt frame, albeit the construction is functorial in the triple $(X,F,\varkappa)$. 

However, Davis, Langer and Zink show in the rank one case, that first of all the obtained morphism is independent of the choice of the Witt frame, and moreover it induces an isomorphism. 

\begin{prop}\label{Prop431} The morphism from Lemma \ref{LemComp} for overconvergent Witt frames is a quasi-isomorphism. The induced morphism (\ref{CompEdgFrames}) is independent of the choice of the Witt frame. \end{prop}
\begin{proof} As in the rank one case (cf. \cite[Proof of Prop. 4.31]{DavisLangerZink}), the independence of the choice of Witt frame follows from the fact, that the construction of the morphism is functorial in $(X,F,\varkappa)$. 

Consequently we may choose a suitable overconvergent Witt frame in order to show that the morphism in question is in fact a quasi-isomorphism. Indeed, according to Prop. \ref{Prop225} there exists an overconvergent Witt frame $(\overline{A},A,\varkappa_F)$ such that the associated morphism
$$H^i_{\text{rig}}(X,\edg)\rightarrow\HH^i(\edg_F\otimes(W^\dagger\Omega_{\overline{A}/k}\otimes\QQ))$$
is an isomorphism.\end{proof}

In particular, the $\edg_F\otimes(W^\dagger\Omega_{\overline{A}/k,\QQ}$ with varying $(X,F,\varkappa)$ are quasi-isomorphic, and hence we omit $F$ in the notation if no confusion arises. 

\subsection{Global results}\label{Subsec3.3}

We follow the argumentation of \cite{DavisLangerZink} to use dagger spaces to globalise this result. They argue that an overconvergent Witt frame (in fact any special frame) $(X,F,\varkappa)$ gives rise to a dagger space structure on the rigid analytic space $]X[_{\widehat{F}}$. We denote it by $]X[^\dagger_{\widehat{F}}$. According to Corollary \ref{Cor313} we have for an overconvergent isocrystal $\edg$
$$R\Gamma(]X[^\dagger_{\widehat{F}},\edg_{]X[_{\widehat{F}}^\dagger}\otimes\Omega_{]X[_{\widehat{F}}^\dagger})=R\Gamma_{\text{rig}}(X,\edg).$$
Moreover, Davis, Langer and Zink point out that the dagger space $]X[_{\widehat{F}}^\dagger$ is functorial in the special frame $(X,F)$. Hence we can rewrite the comparison map in terms of dagger spaces
$$\Gamma(]X[_{\widehat{F}}^\dagger,\edg_{]X[_{\widehat{F}}^\dagger}\otimes\Omega_{]X[_{\widehat{F}}^\dagger})\rightarrow\edg_F\otimes W^\dagger\Omega_{\overline{A}/k,\QQ}$$
and obtain simultaneously a local version via the specialisation map
$$\sP_\ast \edg_{]X[_{\widehat{F}}^\dagger}\otimes\Omega_{]X[_{\widehat{F}}^\dagger}\rightarrow \edg_F\otimes W^\dagger\Omega_{X/k,\QQ}.$$

We come now to the generalisation to an arbitrary smooth quasiprojective scheme $X$ over $k$. This works essentially in the same way as without coefficients. Because $X$ is quasiprojective, we may choose a finite covering $X=\bigcup_{i\in I} X_i$ of $X$ by standard smooth affine schemes over $k$, such that the intersections
$$X_J=X_{i_1}\cap\ldots\cap X_{i_\ell}$$
for a subset $J=\{i_1,\ldots,i_\ell\}\subset I$ are again standard smooth in the sense of \cite[Def. 4.33]{DavisLangerZink}. Denote $X_J=\Spec\overline{A}_J$. It is possible to lift a standard smooth algebra over $k$ to an algebra over $W(k)$ which is again standard smooth. Thus, choose for each $\overline{A}_i$ a standard smooth lift $B_i$ over $W(k)$ and set $F_i=\Spec B_i$ to obtain a special frame $(X_i,F_i)$. Then for any subset $J\subset I$ the closed embedding
$$X_J\rightarrow\prod_{i\in J}F_i=: F_J$$
is a special frame. 

\begin{prop}\label{Prop435} 
Let $\mathcal{Q}=]X_J[_{\widehat{F}_J}^\dagger$ be the dagger space associated to the special frame $(X_J,F_J)$ and $\sP: Q\rightarrow X_J$ the specialisation map. Then the induced morphism
$$\sP_\ast\edg_{\mathcal{Q}}\otimes\Omega^{_{^\bullet}}_{\mathcal{Q}}\rightarrow R\sP_\ast\edg_{\mathcal{Q}}\otimes\Omega^{_{^\bullet}}_{\mathcal{Q}}$$
is a quasi-isomorphism.
\end{prop}
\begin{proof} 
According to the proof of \cite[Prop. 4.35]{DavisLangerZink} the strong fibration theorem for dagger spaces implies that for two liftings $F_1$ and $F_2$ of an affine smooth scheme $Z$ over $k$ such that there is a morphism of frames $\nu:(Z,F_1)\rightarrow (Z,F_2)$ which restricts to the identity on $Z$, the dagger spaces associated to the frames are isomorphic. This ultimately allows us to restrict to the special case, where $F_J$ is of the form $F_{i_0}'\times\aA^n_{W(k)}$ with $F_{i_0}'$ a localisation of a lift $F_{i_0}$ of $X_{i_0}$, for some $i_0\in J$. 

Therefore let $(X_J,F_J)=(X_J,F_{i_0}'\times\aA^n_{W(k)})$ be of this particular form. The associated dagger space to this frame is $\mathcal{Q}\times\breve{D}^n$ where $\breve{D}^n$ designates the open unit ball of dimension $n$ with its natural dagger space structure.  Let $\sP:\mathcal{Q}\times\breve{D}^n\rightarrow Z$ be the specialisation map. As in \cite[Cor. 4.38]{DavisLangerZink} we see that the natural morphism
$$\sP_\ast\edg_{\mathcal{Q}\times\breve{D}^n}\otimes\Omega^{_{^\bullet}}_{\mathcal{Q}\times\breve{D}^n}\rightarrow R\sP_\ast\edg_{\mathcal{Q}\times\breve{D}^n}\otimes\Omega^{_{^\bullet}}_{\mathcal{Q}\times\breve{D}^n}$$
is a quasi-isomorphism. Indeed, consider the spectral sequence of hypercohomology
$$\mathscr{H}^q(R^p \sP_\ast\edg_{\mathcal{Q}\times\breve{D}^n}\otimes\Omega^{_{^\bullet}}_{\mathcal{Q}\times\breve{D}^n})\Rightarrow R^{p+q}\sP_\ast\edg_{\mathcal{Q}\times\breve{D}^n}\otimes\Omega^{_{^\bullet}}_{\mathcal{Q}\times\breve{D}^n}.$$
Because the de Rham complex is locally free, as is by hypothesis the overconvergent isocrystal $\edg$, we may choose an open affine subset $U\subset X_J$ small enough such that for the pre-image $\mathscr{U}\subset{\mathcal{Q}}$ under the specialisation morphism, which is an affinoid dagger space, $\edg_{\mathscr{U}}\otimes\Omega^{_{^\bullet}}_{\mathscr{U}}$ is free. Then the subsequent Lemma \ref{Prop437} shows that the complexes  $H^p(\mathscr{U}\times\breve{D}^n,\edg\otimes\Omega^{_{^\bullet}}_{\mathscr{U}\times\breve{D}^n})$ are acyclic for $p\geqslant 1$. Thus the same holds for the complex $R^p \sP_\ast\edg\otimes\Omega^{_{^\bullet}}_{\mathcal{Q}\times\breve{D}^n}$ and the spectral sequence degenerates. 

This proves that the morphism in question is indeed a quasi-isomorphism and finishes the proof of the proposition.\end{proof}

It remains to show the following statement.
\begin{lem}\label{Prop437} Let $\breve{D}$ be the one dimensional open unit ball with natural dagger space structure. Let $\mathcal{Q}=\Sp^\dagger A$ be a smooth affinoid dagger space, such that $\Omega^1_{\mathcal{Q}}$ is a free $\OO_{\mathcal{Q}}$-algebra and $\edg_{\mathcal{Q}}$ is a free $\OO_{\mathcal{Q}}$-module. 
\renewcommand{\labelenumi}{(\alph{enumi})}\begin{enumerate} \item The canonical morphism
$$H^0(\mathcal{Q},\edg_{\mathcal{Q}}\otimes\Omega^{_{^\bullet}}_{\mathcal{Q}})\rightarrow H^0(\mathcal{Q}\times\breve{D}^n,\edg_{\mathcal{Q}\times\breve{D}^n}\otimes\Omega^{_{^\bullet}}_{\mathcal{Q}\times\breve{D}^n})$$
is a quasi-isomorphism of complexes.
\item The complex $H^1(\mathcal{Q}\times\breve{D}^n,\edg_{\mathcal{Q}\times\breve{D}^n}\otimes \Omega^{_{^\bullet}}_{\mathcal{Q}\times\breve{D}^n})$ is acyclic.
\item $H^i(\mathcal{Q}\times\breve{D}^n,\edg_{\mathcal{Q}\times\breve{D}^n}\otimes\Omega^q_{\mathcal{Q}\times\breve{D}^n})=0$ for $i\geqslant 2$ and all $q$.\end{enumerate}\end{lem}
\begin{proof} The last statement is proved in \cite[Prop. 4.37]{DavisLangerZink} for an arbitrary abelian sheaf which is a coherent $\OO_{\mathcal{Q}\times\breve{D}^n}$-module.

For the remaining two assertions, we may replace the complex $\Omega^{_{^\bullet}}_{\mathcal{Q}\times\breve{D}^n}$ and  $\Omega^{_{^\bullet}}_{\mathcal{Q}}$ respectively in Davis,Langer and Zink's argument by the complex $\edg_{\mathcal{Q}\times\breve{D}^n}\otimes\Omega^{_{^\bullet}}_{\mathcal{Q}\times\breve{D}^n}$ and $\edg_{\mathcal{Q}}\otimes\Omega^{_{^\bullet}}_{\mathcal{Q}}$ which are both locally free modules with integrable connections, and then proceed accordingly. \end{proof}

We are now in a position to finish the proof of the main result.
\begin{theo}\label{Theo3.2.7} Let $X$ be a smooth quasiprojective scheme over $k$, and $\edg\in\Isoc^\dagger(X/W(k)$ an isocrystal. Then there is a natural quasi-isomorphism
$$R\Gamma_{\text{rig}}(X,\edg)\rightarrow R\Gamma(X,\edg\otimes W^\dagger\Omega^{_{^\bullet}}_{X/k\QQ}).$$
\end{theo}
\begin{proof} Choose a finite covering $\{X_i\}_{i\in I}$ as above, and consider the associated simplicial scheme $\mathfrak{X}_{^{_\bullet}}=\{X_J\}_{J\subset I}$ together with the augmentation $\epsilon:\mathfrak{X}_{^{_\bullet}}\rightarrow X$. Upon choosing liftings of the $X_i$ over $W(k)$ one obtains a simplicial object of frames $(\mathfrak{X}_{^{_\bullet}},\mathfrak{F}_{^{_\bullet}})=\{(X_J,F_J)\}_{J\subset I}$ and as a consequence a simplicial object of dagger spaces $\mathfrak{Q}_{^{_\bullet}}=\{\mathcal{Q}_J\}_{J\subset I}$. As seen above, we have for each $J\subset I$ a comparison morphism
$$\sP_\ast\edg_{\mathcal{Q}_J}\otimes\Omega^{_{^\bullet}}_{\mathcal{Q}_J}\rightarrow \edg_{F_J}\otimes W^\dagger\Omega^{_{^\bullet}}_{X_J/k,\QQ}$$
which glues to a morphism of simplicial sheaves
\begin{equation}\sP_\ast\edg_{\mathfrak{Q}_{^{_\bullet}}}\otimes\Omega^{_{^\bullet}}_{\mathfrak{Q}_{^{_\bullet}}}\rightarrow \edg_{\mathfrak{F}_{^{_\bullet}}}\otimes W^\dagger\Omega^{_{^\bullet}}_{\mathfrak{X}_{^{_\bullet}}/k,\QQ}\end{equation}
and by Prop. \ref{Prop435} and \ref{Prop431} it gives rise to a quasi-isomorphism
$$R\sP_\ast\edg_{\mathfrak{Q}_{^{_\bullet}}}\otimes\Omega^{_{^\bullet}}_{\mathfrak{Q}_{^{_\bullet}}}\rightarrow \edg_{\mathfrak{F}_{^{_\bullet}}}\otimes W^\dagger\Omega^{_{^\bullet}}_{\mathfrak{X}_\bullet/k,\QQ}.$$
Because of functoriality in Witt frames, we may now apply the functor $R\epsilon_\ast$ to obtain
\begin{equation}\label{Equ441} R\epsilon_\ast R\sP_\ast\edg_{\mathfrak{Q}_{^{_\bullet}}}\otimes\Omega^{_{^\bullet}}_{\mathfrak{Q}_{^{_\bullet}}}\cong R\epsilon_\ast \edg_{\mathfrak{F}_{^{_\bullet}}}\otimes W^\dagger\Omega^{_{^\bullet}}_{\mathfrak{X}_{^{_\bullet}}/k,\QQ}\cong \edg_{\mathfrak{F}_{^{_\bullet}}}\otimes W^\dagger\Omega^{_{^\bullet}}_{X/k,\QQ}.\end{equation}
As pointed out in the proof of \cite[Thm. 4.40]{DavisLangerZink}, in the case without coefficients a result by Chiarellotto and Tsuzuki \cite{ChiarellottoTsuzuki} implies that the left-hand side of the equality (\ref{Equ441}) is a complex on $X$ whose hypercohomology is rigid cohomology. However, Chiarellotto's and Tsuzuki's results are a lot more general dealing with coherent sheaves with integrable connections. Thus they apply to our situation as well and we obtain
$$R\Gamma( R\epsilon_\ast R\sP_\ast\edg_{\mathfrak{Q}_{^{_\bullet}}}\otimes\Omega^{_{^\bullet}}_{\mathfrak{Q}_{^{_\bullet}}})\cong R\Gamma_{\text{rig}}(X,\edg)$$
which proves the theorem.
\end{proof}



\addcontentsline{toc}{section}{References}

\vspace{1cm}
\hrule
\vspace{.5cm}
{\noindent
\textsc{Universit\"at Regensburg}\\     	
Fakult\"at f\"ur Mathematik	\\
Universit\"atsstra{\ss}e 31	\\
93053 Regensburg		\\
Germany		      	\\
(+ 49) 941-943-2664       	\\
\verb=veronika.ertl@mathematik.uni-regensburg.de=\\
\url{http://www.mathematik.uni-regensburg.de/ertl/} \\}

\end{document}